\documentclass[12pt]{article}
\title{Zero-free regions inspired by work of Heath-Brown}
\author{Chiara Bellotti,\; Tim Trudgian\footnote{Supported by Australian Research Council Discovery Project DP240100186.} \; and Andrew Yang \\
School of Science, UNSW Canberra, Australia \\
c.bellotti@unsw.edu.au\\ timothy.trudgian@unsw.edu.au \\  andrew.yang1@unsw.edu.au}
\usepackage{indentfirst}
\usepackage{url}
\usepackage{enumerate}
\usepackage{amsthm}
\usepackage{amsmath,mathrsfs}
\usepackage{comment}
\usepackage{fullpage}
\usepackage{amssymb}
\usepackage{booktabs}

\newtheorem{thm}{Theorem}

\newtheorem{lemma}{Lemma}
\newtheorem{definition}{Definition}
\newtheorem{remark}{Remark}

\usepackage[dvipsnames]{xcolor}
\usepackage{todonotes}

\begin{document}
\maketitle
\begin{center}
\noindent\textit{This work we dedicate\\
To dear D.R.H.-B.\\
To help to celebrate\\
His diamond jubilee.}
\end{center}
\begin{abstract}
\noindent
We prove a new explicit zero-free region for the Riemann zeta-function, drawing substantially on Heath-Brown's seminal work on Linnik's constant. Using these ideas we are able to prove that $\zeta(\sigma + it)\ne 0$ whenever $t\geq 3$ and $\sigma \geq 1- 1/(4.896\log t)$. 
\end{abstract}


\section{Introduction}
One of the standard ways to show that $\zeta(s)$, the Riemann zeta-function, is non-zero in a certain region is to use the `zero-detector' arising from its logarithmic derivative. Specifically, for $s = \sigma + it$ and for $\sigma >1$ we have $-\zeta'(s)/\zeta(s) = \sum_{n=1}^{\infty} \Lambda(n)n^{-s}$, where $\Lambda(n)$ is the von Mangoldt function. One manipulates the Dirichlet series and uses the fact that $\zeta(s)$ has a pole at $s=1$ to show that there are no zeroes on the line $\sigma=1$ --- see, e.g., \cite[Chapter 6.1]{MV}. Such ideas can be extended to $\sigma<1$ leading to the `classical' zero-free region. This is the region $\sigma \geq 1 - A/\log t$, for sufficiently large $t$ and for some positive constant $A$, for which $\zeta(\sigma + it)\neq 0$. This was first proved by de la Vall\'{e}e Poussin in 1899: see \cite{HoffTrudgian} for a summary of improvements to the constant $A$. While asymptotically sharper zero-free regions are known (we mention two standard results in (\ref{Ybound}) and (\ref{Bbound})), many applications rely on sharp numerical estimates for the classical zero-free region. This is the focal point of our article. We shall be concerned with showing
\begin{equation}\label{rooster}
\zeta(\sigma + it) \ne 0 \qquad  \textrm{in the region} \qquad  \sigma > 1 - \frac{A}{\log t}, \qquad t\geq 3,
\end{equation}
for a small constant $A$.

Two key ideas have been used to enlarge the classical zero-free region of $\zeta(s)$. The first idea, due to Stechkin \cite{stechkin_zeros_1970}, was to make use of the fact that the complex zeroes of the zeta-function exist in pairs about the critical line. The second idea, due to Heath-Brown \cite{heath_brown_zero_1992}, was to smooth the classical zero-detector using a specific weight function. The zero-detectors they used have the Dirichlet series representations
\[
\sum_{n \ge 1}\frac{\Lambda(n)}{n^s}\left(1 - \frac{\kappa}{n^\delta}\right)\qquad \text{and}\qquad\sum_{n \ge 1}\frac{\Lambda(n)}{n^s}f(\log n),
\]
respectively, where $f$ is a compactly supported, twice-differentiable real-valued function that decays steadily to 0.

Remarkably, the two ideas appear to be largely orthogonal, despite their both being weighted versions of the classical zero-detector $\sum_{n \ge 1}\Lambda(n) n^{-s}$. This is because the weight functions are used in two different ways. 
Stechkin's function attenuates terms in the zero-detector corresponding to small $n$, while Heath-Brown's function attenuates terms corresponding to large $n$. 

One may therefore wonder if the two ideas can be used concurrently. Indeed, Kadiri \cite{Kadiri} showed that this can be done profitably by using the zero detector
\begin{equation}\label{kadiri_test_function_intro}
\sum_{n \ge 1}\frac{\Lambda(n)}{n^s}f(\log n)\left(1 - \frac{\kappa'}{n^{\delta'}}\right)
\end{equation}
for some constants $\kappa', \delta' > 0$, and where $f$ is the function considered by Heath-Brown. Although \cite{Kadiri} contains several  innovations, the combination of smoothing functions explains most of the improvement in (\ref{rooster}), from $A = (8.463)^{-1}$ in \cite{ford_zero_2002} to $A = (5.69693)^{-1}$ in \cite{Kadiri}.

Our goal is to combine these two ideas more efficiently to prove the following.

\begin{thm}\label{intro_new_theorem1}
If $t \ge 3$ then $\zeta(\sigma + it) \ne 0$ in the region $\sigma > 1 - 1/(4.896 \log t).$
\end{thm}
The constant of $4.896$ in Theorem \ref{intro_new_theorem1} improves on the previous best $5.559$ from \cite{mossinghoff_explicit_2022}.
We prove this first in a restricted range below.
\begin{lemma}\label{classical_main_lem}
If $3 \le t \le \exp(76.47)$ then $\zeta(\sigma + it) \ne 0$ in the region $\sigma > 1 - 1/(4.896 \log t).$
\end{lemma}
We now show how Theorem \ref{intro_new_theorem1} follows from Lemma \ref{classical_main_lem}. First we quote two zero-free regions that appear to be the strongest in the literature in various ranges. The third author \cite{Yang} proved an explicit version of Littlewood's zero-free region
\begin{equation}\label{Ybound}
\zeta(\sigma + it) \ne 0 \qquad  \textrm{in the region} \qquad  \sigma > 1 - \frac{\log\log t}{21.233\log t}, \qquad t\geq 3.
\end{equation}
The first author \cite{BZero} proved an explicit version of the Vinogradov--Korobov zero-free region
\begin{equation}\label{Bbound}
\zeta(\sigma + it) \ne 0 \qquad  \textrm{in the region} \qquad  \sigma > 1 - \frac{1}{53.989(\log t)^{2/3} (\log\log t)^{1/3}},\qquad t\geq 3.
\end{equation}
We note that (\ref{Bbound}) is sharper than (\ref{Ybound}) for $t> \exp(482036)$. For the proof of Theorem \ref{intro_new_theorem1}, however, the relevant point is that the explicit Littlewood zero-free region (\ref{Ybound}) is stronger than our Theorem \ref{intro_new_theorem1} for $t> \exp(76.463)$. Since Lemma \ref{classical_main_lem} establishes Theorem \ref{intro_new_theorem1} for $2 \leq t \leq \exp (76.47)$, the zero-free region (\ref{Ybound}) covers the remaining range $t>\exp (76.463)$, and hence Theorem \ref{intro_new_theorem1}  follows.

The constants in (\ref{Bbound}) and (\ref{Ybound}) were improved in the third author's thesis \cite{Yangthesis}, where $21.233$ and $53.989$ are replaced by $19.62$ and $51.34$, respectively. In particular, the improved Littlewood zero-free region would yield the following stronger classical zero-free region.
\begin{thm}\label{intro_new_theorem2}
If $t \ge 3$ then $\zeta(\sigma + it) \ne 0$ in the region $\sigma > 1 - 1/(4.8594 \log t).$
\end{thm}
As before, this theorem would from a result restricted to a certain range.
\begin{lemma}\label{classical_main_lem2}
If $3 \le t \le \exp(56.693)$ then $\zeta(\sigma + it) \ne 0$ in the region $\sigma > 1 - 1/(4.8594 \log t).$
\end{lemma}
Indeed, the improved Littlewood zero-free region is sharper than the bound claimed in Theorem \ref{intro_new_theorem2} for every $t> \exp (56.691)$. Therefore, once Lemma \ref{classical_main_lem2} is established in the range $3 \leq t \leq \exp (56.693)$, the improved Littlewood zero-free region covers the remaining range of $t$, and Theorem \ref{intro_new_theorem2} would follow.

The outline of this paper is as follows. In \S \ref{moat} we motivate the choice of our new smoothing function $f$. In \S \ref{Choice} we select our function $f$. In \S \ref{sec:trig_poly} we select an appropriate trigonometric polynomial to enhance the size of the zero-free region. Finally, in \S \ref{castle} we prove Lemma \ref{classical_main_lem}. 

\section{Reasons for choosing $f$}\label{moat}
We now motivate our choice of a smoothing function $f$, which leads to our improved region in Theorem \ref{intro_new_theorem1}. We draw heavy inspiration from the argument of Heath-Brown \cite{heath_brown_zero_1992}. 

Since non-trivial zeroes are symmetric about the critical line and the real axis, we focus on the quadrant $\{z \in \mathbb{C}: \Re\,z \ge 1/2, \Im\,z \ge 0\}$. In fact, we can further restrict to 
\begin{equation}\label{classical_H_defn}
\Im\,z \ge H := 3\cdot 10^{12}
\end{equation}
since all zeroes $\rho = \beta + i\gamma$ with $0 < \gamma \le H$ have been computationally verified to lie on the critical line \cite{platt_riemann_2021}. The non-trivial zeroes at heights just above $H$ pose the greatest difficulty. 

Suppose $f$ is a non-negative, bounded and compactly supported function with a well-defined Laplace transform $F(z) := \int_0^{\infty}e^{-zu}f(u)\text{d}u.$
Under suitable assumptions, the associated zero-detector has an explicit formula
\begin{equation}\label{rough_1}
\Re \sum_{n \ge 1}\frac{\Lambda(n)}{n^s}f(\log n) = \frac{f(0)}{2}\Re\frac{\Gamma'}{\Gamma}\left(\frac{s}{2} + 1\right) + \Re F(s - 1) 
- \sum_{\rho}\Re F(s - \rho) + o(1),
\end{equation}
where $s = \sigma + it$ and the error term is taken with respect to $t \to \infty$. This formula may be viewed as a generalisation of Hadamard's product formula used in the classical proof, with an important distinction that the formula remains valid for $s$ \textit{inside} the critical strip. 

It turns out that moving $s$ inside the critical strip is accretive to the size of the zero-free region, provided that $s$ itself resides in a zero-free region. This inspires an interesting inductive argument: one begins with a known zero-free region, then iteratively\footnote{As usual, there is no free lunch --- eventually the marginal progress made with each iteration tends to 0 and the size of the zero-free region stabilises.} expands it by taking $s$  deeper inside the critical strip as permitted by the new zero-free region at each stage. This idea was introduced in \cite{Kadiri} and has since become standard in the literature. 

Specifically, suppose that there exists a simple\footnote{Since  zeroes in the sum \eqref{rough_1} are counted with multiplicity, if $\rho_0$ has multiplicity $m$ then it is included $m$ times, and we arbitrarily rename one of them $\rho_0$. The remaining $m - 1$ terms are treated the same way as the other zeroes $\rho \ne \rho_0$. }, complex zero $\rho_0 = \beta_0 + it$ lying just outside a known zero-free region $\sigma > 1 - A/\log t$ for some constant $A > 0$. We seek to show that $\rho_0$ actually satisfies a stronger inequality, of the form $\beta_0 \le 1 - (A + \varepsilon)/\log t$ say, for some $\varepsilon > 0$. Assume for a contradiction that
\begin{equation}\label{rough_assumption_beta}
A \le (1 - \beta_0)\log t < A + \varepsilon.
\end{equation}
We consider a non-negative trigonometric polynomial 
$P(\theta) = \sum_{0 \le k \le K}a_k \cos(k\theta)$,
satisfying $P(\theta) \ge 0$, $a_k > 0$ and $a_1 > a_0$, which generalises the polynomial $3 + 4\cos\theta + \cos2\theta$ used in the classical proof --- see, e.g., \cite[Chapter 3]{Titchmarsh}. The existence and construction of such polynomials is non-trivial, and we refer the reader to \cite{HoffTrudgian} for an account. Combined with \eqref{rough_1} evaluated at $s = s_k := \sigma + ikt$ for $0 \le k \le K$, we obtain
\begin{equation}\label{rough_main_ineq}
\begin{split}
0 &\le \sum_{n \ge 1}\frac{\Lambda(n)}{n^\sigma}f(\log n)P(t\log n) = \Re\sum_{0 \le k \le K}a_k \sum_{n \ge 1}\frac{\Lambda(n)}{n^{s_k}}f(\log n)\\
&= \Re\sum_{0 \le k \le K}a_k \left(\frac{f(0)}{2}\frac{\Gamma'}{\Gamma}\left(\frac{s_k}{2} + 1\right) + F(s_k - 1) - \sum_{\rho}F(s_k - \rho)\right) + o(1).
\end{split}
\end{equation}
Using classical estimates, as $t \to\infty$,
\[
\Re\sum_{0 \le k \le K}a_k \left(\frac{f(0)}{2}\frac{\Gamma'}{\Gamma}\left(\frac{s_k}{2} + 1\right) + F(s_k - 1)\right) < c f(0) \log t + a_0 F(\sigma - 1),
\]
for an absolute constant $c > 0$. To bound the sum over zeroes we  assume the following positivity condition: for any non-trivial zero $\rho$,
\begin{equation}\label{one_signedness_condition}
\Re F(s - \rho) + \Re F(s - (1 - \overline{\rho})) \ge 0.
\end{equation}
Note that both $\rho$ and $1 - \overline{\rho}$ are (possibly distinct) complex zeroes, so \eqref{one_signedness_condition} may be viewed as a slightly ``averaged" version of the pointwise condition $\Re F(s - \rho) \ge 0$ considered in Heath-Brown \cite{heath_brown_zero_1992}. Condition \eqref{one_signedness_condition} is also satisfied by Kadiri's \cite{Kadiri} smoothing function: it is here that the ideas of Stechkin and Heath-Brown are combined. This condition allows us to isolate the contribution of just one pair of zeroes, say $\beta_0 + it$ and $1 - \beta_0 + it$, since 
\[
\Re\sum_{\rho}F(s_1 - \rho) \ge F(\sigma - \beta_0) + F(\sigma - 1 + \beta_0)
\]
and $
\Re\sum_{\rho}F(s_k - \rho) \ge 0$ when $k \ne 1$.
The $F(\sigma - 1 + \beta_0)$ term is $o(1)$, so combining these estimates with \eqref{rough_main_ineq}, one has 
\begin{equation}\label{classical_main_ineq_5}
c\log t > \frac{a_1 F(\sigma - \beta_0) - a_0 F(\sigma - 1) + o(1)}{f(0)}.
\end{equation}
Next we choose $\sigma$ and $f$ so that the right side is $\gg 1/\eta$, where $
\eta := 1 - \beta_0.$
This can be achieved, for instance, if one takes $\sigma = 1 - (A + o(1))/\log t$ and $f(u) = \eta g(\eta u)$ for some fixed function $g$. With these choices the right side of \eqref{classical_main_ineq_5} is
\[
\frac{a_1 G(1 - \mu) - a_0 G(-\mu)}{g(0)} \eta^{-1}
\]
where $G$ is the Laplace transform of $g$ and $\mu = (1 - \sigma)/\eta = 1 + O(\varepsilon)$. Taking $\varepsilon \to 0$, the coefficient of $\eta^{-1}$ is a constant, as required. Therefore, one arrives at a contradiction to \eqref{rough_assumption_beta} if $A$ is small enough, which completes the argument.

Let us now consider the problem of choosing the best $f$ for each $\eta > 0$ (we now regard $\eta$ as fixed). First, we review the constraints one needs to place on $f$. To control the various $o(1)$ error terms we require $F(z) = f(0)/z + O(|z|^{-2})$ as $|z| \to \infty$. This condition is surprisingly mild\footnote{In fact, our choice of $f$ satisfies $F(z) = f(0)/z + O(|z|^{-3})$. The improved error was also noted in \cite{Kadiri}.} --- in \cite{heath_brown_zero_1992} this was enforced by requiring $f$ to be compactly supported and having a bounded and continuous second derivative. The second, more important condition comes from \eqref{one_signedness_condition} --- for this it suffices to require that $\Re F(z) + \Re F(2\sigma - 1 - \overline{z}) \ge 0$ for all $\Re\,z \ge 0$. 

Taking $\varepsilon \to 0$ and ignoring $o(1)$ terms, from \eqref{classical_main_ineq_5} the optimisation problem we are led to consider is (for each fixed $\eta$)
\begin{equation}\label{classical_optimization_problem}
\max_{f} \frac{a_1 F(0) - a_0 F(-\eta)}{f(0)}
\end{equation}
subject to $F(z) = f(0)/z + O(|z|^{-2})$ and
\begin{equation}\label{classical_one_signedness_g}
\Re F(z) + \Re F(2\sigma - 1 - \overline{z}) \ge 0\qquad(\Re\,z \ge 0).
\end{equation}
In \cite{heath_brown_zero_1992}, Heath-Brown studied a similar optimisation problem, with the constraint $\Re F(z) \ge 0$ in place of \eqref{classical_one_signedness_g}. In the next section we explore how solutions to Heath-Brown's optimisation problem can be adapted to solve \eqref{classical_optimization_problem}.

\section{The test function}\label{Choice}
In this section we choose the test function $f$. 
First, suppose that $w$ is a twice-differentiable, compactly-supported, real-valued function. Let $W(z)$ denote its Laplace transform, which satisfies $
\Re W(z) \ge 0$ when $\Re\,z \ge 0$.
Examples of such functions were studied by Heath-Brown~\cite{heath_brown_zero_1992}. One particularly relevant function is \cite[Lemma 7.1]{heath_brown_zero_1992} with $\lambda = 1$, given by:
\begin{definition}[The functions $w$ and $W$]\label{classical_wW_defn}
For fixed $\theta \in (0, \pi/2)$, let $w$ be defined as 
\begin{align*}
w(u) &:= \sec^2\theta(\sec^2\theta(\theta\cot\theta - u/2)\cos(u\tan\theta) + 2\theta\cot\theta - u \\
&\qquad\qquad + \sin(2\theta - u\tan\theta) \csc 2\theta - 2(1 + \sin(\theta - u\tan\theta)\csc\theta)),
\end{align*}
for $0 \le u \le 2\theta\cot\theta$, and $w(u) := 0$ elsewhere. Let $W(z) := \int_0^\infty e^{-zu} w(u)\text{d}u$ denote the Laplace transform of $w$.
\end{definition}
In our notation, one choice of $f$ considered by Heath-Brown was 
\begin{equation}\label{test_f_heathbrown}
\eta w(\eta u).
\end{equation}
In \cite{heath_brown_zero_1992}, it was shown via a calculus-of-variations argument that this choice of $f$ solves, within a large family of functions, the optimisation problem
\[
\max_f \frac{a_1 F(0) - a_0 F(-\eta)}{f(0)}\qquad\text{subject to}\qquad \Re F(z) \ge 0 \quad (\Re\,z \ge 0).
\]
Consider now the function
\begin{equation}
\label{f_hypothetical_defn}
f(u) = \eta w(\eta u)h(u),\qquad h(u) := \frac{1}{1 + e^{-(2\sigma - 1)u}}.
\end{equation}
This choice of $f$ satisfies \eqref{classical_one_signedness_g}, since for any $z = iy$ ($y \in \mathbb{R}$),
\begin{align}
\Re F(z) + \Re F(2\sigma - 1 - \overline{z}) &= \Re \int_0^{\infty}e^{-iyu}(1 + e^{-(2\sigma - 1)u})f(u)\text{d}u \label{hypothetical_W_ineq_1}\\
&= \Re\int_0^{\infty}e^{-iyu/\eta}w(u)\text{d}u \label{hypothetical_W_ineq_2}\\
&= \Re W(iy/\eta)\ge 0.\notag
\end{align}
We can then extend this inequality to hold for all $\Re\,z \ge 0$ via a result of Heath-Brown \cite[Lemma 4.1]{heath_brown_zero_1992}.
In particular, one might expect that this choice of weight function $h$ is optimal because there is no ``loss" in passing from \eqref{hypothetical_W_ineq_1} to \eqref{hypothetical_W_ineq_2}.

Unfortunately, $f$ does not satisfy the condition $F(z) = f(0)/z + O(|z|^{-2})$. Furthermore, the inverse Laplace transform of $h$ does not converge, which makes analysis difficult. For these reasons, we will instead use a discrete approximation of $h(u)$ which produces numerically similar results, and which possesses the desired properties. Roughly speaking, we choose parameters $\kappa_m$ ($1 \le m \le M$) so that uniformly for $x \in [0, 1]$, one has
\begin{equation}\label{kappa_m_inspiration}
\sum_{0 \le m \le M} \kappa_m x^m \approx \frac{1}{1 + x}.
\end{equation}
One may be inclined to find the coefficients $\kappa_m$ in \eqref{kappa_m_inspiration} using a Taylor expansion at say $x = 1/2$; however, this leads to approximations that are catastrophically bad at, e.g., $x = 0$. Instead, we search for favourable coefficients computationally. As a first pass we choose $\kappa_m$ to minimise the $L^\infty$ norm 
\[
\max_{0 \le x \le 1}\bigg|\frac{1}{1 + x} - \sum_{0 \le m\le M}\kappa_m x^m \bigg|.
\]
We perturb the coefficients so as to make the arguments that follow simpler (in particular, we opt for slightly less favourable coefficients that can be more easily analysed, and whose rational approximations are ``simple" to aid our presentation). We take $M = 6$ and
\begin{equation}\label{kappa_delta_defn}
\begin{split}
[\kappa_m]_{0 \le m \le M} &= 
\left[1, \; -\frac{851}{859},\; \frac{780}{859},\; -\frac{525}{859},\; \frac{171}{859},\; \frac{28}{859},\; -\frac{29}{859}\right].
\end{split}
\end{equation}
Our choice of $f$ is then given by 
\begin{definition}[The functions $f$ and $F$]\label{wfF_defn}
For fixed $\eta > 0$ and $1/2 < \sigma < 1$, define
\begin{equation}\label{f_defn_eqn}
f(u) := \eta w(\eta u) \sum_{0 \le m \le M}\kappa_m e^{-(2\sigma - 1)mu}
\end{equation}
where $M$ and $\kappa_m$ are fixed parameters defined in \eqref{kappa_delta_defn}. Let $F(z) := \int_0^\infty e^{-zu} f(u)\text{d}u$ denote Laplace transform of $f$.
\end{definition}
In comparison, Kadiri's \cite{Kadiri} test function \eqref{kadiri_test_function_intro} may be expressed in our notation as 
\begin{equation}\label{test_f_kadiri}
\eta w(\eta u)\sum_{0 \le m \le 1}\kappa_m e^{-\delta_m u},
\end{equation}
with $\kappa_0 = 1$, $\delta_0 = 0$, $\kappa_1 = \kappa'$ and $\delta_1 = \delta'$. 
One can measure the ``efficiency" of the test functions \eqref{test_f_heathbrown}, \eqref{f_hypothetical_defn}, \eqref{f_defn_eqn} and \eqref{test_f_kadiri} by evaluating \eqref{classical_optimization_problem}, which in each case gives
\[
\frac{a_1 F(0) - a_0 F(-\eta)}{f(0)} = (c + O(\eta))\frac{a_1 W(0) - a_0 W(-1)}{\eta w(0)},
\]
where $c = 859/433 = 1.98\ldots$ in the case of \eqref{f_defn_eqn}, $c = 2$ for the hypothetical function \eqref{f_hypothetical_defn}, $c = 1.78\ldots$ for \eqref{test_f_kadiri} and $c = 1$ for \eqref{test_f_heathbrown}. In this sense, the value attained by our discrete approximation is within one percent of the hypothetical limit of the method.
Let us now record some properties of these functions.
\begin{lemma}
One has $f(u) \ge 0$ for all $u \ge 0$. 
\end{lemma}
\begin{proof}
By inspection $w(u) \ge 0$, so upon taking $x = e^{-(2\sigma - 1)u}$ (and noting that the exponent is negative from the assumption $\sigma > 1/2$), it suffices to show that 
\[
p(x) = \sum_{0 \le m \le M}\kappa_m x^m > 0\qquad (0 < x \le 1),
\]
which we verify by numerically isolating all $M$ roots of the polynomial $p$, checking that none resides on $(0, 1]$, and noting that $p(1) > 0$. 
\end{proof}

Next, we record an explicit estimate for $W(z)$, due to Ford \cite{ford_zero_2002, ford_zero_2022}.
\begin{lemma}[Ford \cite{ford_zero_2002} \S 7]\label{W0_bound_lem}
Let $\theta \in (0, \pi/2)$ and $r > \tan\theta$ be fixed, and let $W$ be defined in Definition \ref{classical_wW_defn}. If we write
\[
W(z) = \frac{w(0)}{z} + W_0(z),
\]
then for all $\Re\,z \ge \nu$ and $|z| \ge r$ one has $|W_0(z)| \le C(\nu, r)|z|^{-3}$,  where 
\[
C(\nu, r) := c_0 r \frac{c_2(r + 1)^2(e^{-2\nu \theta\cot\theta} + 1) + c_1r + c_3r^3}{(r^2 - \tan^2\theta)^2}
\]
and
\[
c_0 = \csc\theta\sec^2\theta,\qquad c_1 = (\theta - \sin\theta \cos\theta)\tan^4\theta,
\]
\[
c_2 = \tan^3\theta \sin^2\theta,\qquad c_3 = (\theta - \sin\theta\cos\theta)\tan^2\theta.
\]
\end{lemma}

Before proceeding further, in Definitions \ref{classical_wW_defn} and \ref{wfF_defn} let us fix
$\theta = 1.1338,$
the choice of which was determined via numerical experimentation. We therefore have
\[
w(0) = \sec^2\theta (\theta \tan\theta + 3 \theta\cot\theta - 3) = 5.672787598\ldots.
\]
Recall $H$ is defined in \eqref{classical_H_defn}. For future use we also define the following parameters
\begin{equation*}\label{classical_param_defn}
\begin{split}
T_0 &= 10^{10},\\
\delta_m &= (2\sigma - 1)m,\\
\eta_0 &= A_0/\log H,\\
K &= 16,\\
\sigma_0 &= 1 - A_0/\log(KH + T_0),
\end{split}
\end{equation*}
where $A_0$ is from Lemma \ref{classical_main_lem}  or Lemma \ref{classical_main_lem2}. We have defined $\sigma_0$ so that our eventual choice of $\sigma$ satisfies $\sigma \in [\sigma_0, 1)$, and similarly one is guaranteed that $\eta \in (0, \eta_0]$. The parameter $K$ is the degree of the non-negative trigonometric polynomial which we discuss in \S\ref{sec:trig_poly}. 
\begin{table}[ht]
\centering
\begin{tabular}{|c|c|c|}
\hline
&Lemma \ref{classical_main_lem} & Lemma \ref{classical_main_lem2}\\\hline
$A_0$ & $(4.8596)^{-1}$& $(4.8594)^{-1}$\\\hline
$\eta_0$ & $0.0071093\dots$  &$0.0071628\dots$\\ \hline
$\sigma_0$ & $0.9935164\dots$ & $0.9934675\dots$\\\hline
\end{tabular}
\end{table}

Recall that $\Re W(z) \ge 0$ for $\Re\,z \ge 0$ (see \cite{heath_brown_zero_1992}). This property is used to show that $f$ satisfies the non-negativity condition \eqref{one_signedness_condition}.

\begin{lemma}[Non-negativity property]\label{f_one_sign_lem}
Let $\sigma \in [\sigma_0, 1)$, $\eta \in (0, \eta_0]$ and $F(z)$ be as defined in Definition \ref{wfF_defn}. Then for all complex numbers $z$ with $0 \le \Re\,z \le 2\sigma - 1$, one has
\begin{equation}\label{f_one_sign_ineq_1}
\Re F(z) + \Re F(2\sigma - 1 - \overline{z}) \ge 0.
\end{equation} 
Furthermore, for all $-\eta \le \Re\,z \le 1$ and $|\Im z| \ge T_0$, one has
\begin{equation}\label{f_one_sign_ineq_2}
\Re F(z) + \Re F(2\sigma - 1 - \overline{z}) > -\frac{44}{|\Im z|^{2}}\eta. 
\end{equation}
\end{lemma}
\begin{proof}
Since $\Re F(2\sigma - 1 - \overline{z}) = \Re F(2\sigma - 1 - z)$, we may show \eqref{f_one_sign_ineq_1} by proving $\Re F(z) + \Re F(2\sigma - 1 - z) \ge 0$ on the boundary of the rectangle $\{z \in \mathbb{C}: 0 \le \Re\,z \le 2\sigma - 1, |\Im z| \le T\}$ for every sufficiently large $T$. The result then holds by the maximum modulus principle (applied to the holomorphic function $e^{-F(z) - F(2\sigma - 1 - z)}$).

First we will show that the desired inequality holds on the line $\Re\,z = 0$ (so that by symmetry it also holds for $\Re\,z = 2\sigma - 1$). Let us write
\[
(1 + x)\sum_{0 \le m \le M}\kappa_m x^m = \sum_{0 \le m \le M + 1}b_m x^m.
\]
Then
\begin{equation}\label{classical_f_one_sign_main_ineq}
\begin{split}
\Re F(iy) &+ \Re F(2\sigma - 1 - iy) \\
&= \Re\int_0^{\infty}e^{-iyu}(1 + e^{-(2\sigma - 1)u})\eta w(\eta u)\sum_{0 \le m \le M}\kappa_m e^{-\delta_m u}\text{d}u\\
&= \Re \sum_{0 \le m \le M + 1}b_m \int_0^{\infty}e^{-(\delta_m + iy)u/\eta}w(u) \text{d}u\\
&\ge \Re\sum_{1 \le m \le M + 1}b_m W\left(\frac{\delta_m + iy}{\eta}\right),
\end{split}
\end{equation}
where we dropped the term corresponding to $m = 0$ since $\Re  W(iy/\eta) \ge 0$. If $z = (\delta_m + iy)/\eta$ with $m \ge 1$, then $|z| \ge \Re\,z \ge (2\sigma_0 - 1)/\eta_0 > 138$. Thus, applying Lemma \ref{W0_bound_lem},
\begin{align*}
\frac{1}{(2\sigma - 1)w(0)}\left|W_0\left(\frac{\delta_m + iy}{\eta}\right)\right| \le \frac{\eta_0^2 C(138, 138)}{(2\sigma_0 - 1)^2 w(0)}\frac{1}{m}\frac{\eta}{|\delta_m + iy|^2} < \frac{\varepsilon_0}{m}\frac{\eta}{|\delta_m + iy|^2} 
\end{align*}
with $\varepsilon_0 = 1/2000$, and also
\[
\frac{1}{(2\sigma - 1)w(0)}\Re \frac{w(0)}{(\delta_m + iy)/\eta} = m\frac{\eta}{|\delta_m + iy|^2}.
\]
Therefore,
\begin{align*}
\frac{1}{(2\sigma - 1)w(0)\eta}&\Re \sum_{1 \le m \le M + 1}b_m W\left(\frac{\delta_m + iy}{\eta}\right) \\
&\ge \sum_{1 \le m \le M + 1}\frac{b_m m - |b_m|\varepsilon_0/m}{|\delta_m + iy|^2} \ge B(y) := \sum_{1 \le m \le M + 1}b_m B_m(y)
\end{align*}
where, since $c_0 m < (2\sigma_0 - 1)m \le \delta_m \le m$, where $c_0 = 151/153$,
\[
B_m(y) := \begin{cases}
\dfrac{m + \varepsilon_0/m}{c_0^2 m^2 + y^2},& b_m \le 0,\\
\dfrac{m - \varepsilon_0/m}{m^2 + y^2},& b_m > 0.
\end{cases}
\]
It suffices to verify that $B(y) \ge 0$ for all $y \in \mathbb{R}$. Substituting the values of $\kappa_m$ ($0 \le m\le M$), one may explicitly derive (with the aid of computer assistance) that $B(y) = p(y)/q(y)$, where
\begin{align*}
p(y) &= 4061245152630328137981 y^{12} + 4077560173170236734684710 y^{10} \\
&- 104378137212291977844887868 y^{8} - 4484512641017853031179075270 y^{6} \\
&+ 135673322742635307737680349343 y^{4} - 229732179325278720034298507440 y^{2} \\
&+ 112359769561546903428467326544,
\end{align*}
\begin{align*}
q(y) &= 4012454647232553285540000 (y^{2} + 1)(y^{2} + 9)(y^{2} + 25)\\
&\qquad\qquad\times (y^{2} + 4c_0^2)(y^{2} + 16c_0^2)(y^{2} + 36c_0^2)(y^{2} + 49c_0^2).
\end{align*}
By inspection, $q(y)$ has no real roots, and by computationally isolating all twelve complex roots of $p(y)$, we verify that none are real. Therefore, $B(y)$ is one-signed, and so, since $B(0) > 0$ we have $B(y) > 0$ for all $y \in \mathbb{R}$. Combined with \eqref{classical_f_one_sign_main_ineq}, one obtains \eqref{f_one_sign_ineq_1} on $\Re\,z = 0$.

Next we show \eqref{f_one_sign_ineq_1} holds on the line segment $\Im z = T$ ($0 \le \Re\,z \le 2\sigma -1$) for sufficiently large $T$ (and hence also on $\Im z = -T$ since $\Re F(z) = \Re F(\overline{z})$). For this we simply note that for any fixed $\eta$ and uniformly for $0 \le x \le 2\sigma - 1$,
\[
\Re F(x + iT) = \Re \sum_{0 \le m \le M}\kappa_m W\left(\frac{x + iT}{\eta}\right) = \frac{\eta w(0)}{T^2}(\kappa + o(1))\qquad (T \to \infty),
\]
where $\kappa := \sum_{0 \le m \le M}\kappa_m > 0$, so $\Re F(x + iT) > 0$ for sufficiently large $T$. Thus \eqref{f_one_sign_ineq_1} holds. 

Now consider \eqref{f_one_sign_ineq_2}. In light of the bound \eqref{f_one_sign_ineq_1} it suffices to show \eqref{f_one_sign_ineq_2} for $-\eta \le \Re\,z < 0$ and $\Im z \ge T_0$. First, suppose $u = x + iy$ is a complex number and $v$ a real number, such that $|v| \le \eta$, $x \ge 0$ and $|y| \ge T_0$. Then, by Lemma \ref{W0_bound_lem} one has
\[
W\left(\frac{u + v}{\eta}\right) - W\left(\frac{u}{\eta}\right) = \eta w(0)\left(\frac{1}{u + v} - \frac{1}{u}\right) + W_0\left(\frac{u + v}{\eta}\right) - W_0\left(\frac{u}{\eta}\right),
\]
where, since $\Re (u/\eta) \ge 0$, $\Re (u + v)/\eta \ge -1$ and $|y/\eta| \ge T_0/\eta_0$ one has 
\[
\left|W_0\left(\frac{u + v}{\eta}\right)\right| \le \frac{\eta_0^2 C(-1, T_0/\eta_0)}{T_0}\frac{\eta}{y^2},\quad \left|W_0\left(\frac{u}{\eta}\right)\right| \le \frac{\eta_0^2 C(0, T_0/\eta_0)}{T_0}\frac{\eta}{y^2}.
\]
Meanwhile, 
\[
\left|\frac{1}{u + v} - \frac{1}{u}\right| = \frac{|v|}{|u||u + v|} \le \frac{\eta}{y^2}.
\]
Combining everything, 
\[
\left|W\left(\frac{u + v}{\eta}\right) - W\left(\frac{u}{\eta}\right)\right| \le \frac{\eta}{y^2}\left(w(0) + \frac{C(-1, T_0/\eta_0) + C(0, T_0/\eta_0)}{T_0}\eta_0^2\right) < \frac{5.7\eta}{y^2}.
\]
Taking $u = \delta_m + iy$ ($m \ge 0$) and $v = \Re\,z$, one gets
\[
\Re F(z) = \sum_{0 \le m \le M}\kappa_m W\left(\frac{z + \delta_m}{\eta}\right) > \Re F(iy) - \frac{5.7\eta}{y^2}\sum_{0 \le m \le M}|\kappa_m|.
\]
Similarly, if we instead take $u = 2\sigma - 1 + \delta_m + iy$ and $v = -\Re\,z$, then 
\[
\Re F(2\sigma - 1 - \overline{z}) > \Re F(2\sigma - 1 + iy) - \frac{5.7\eta}{y^2}\sum_{0 \le m \le M}|\kappa_m|.
\]
Adding the two inequalities gives \eqref{f_one_sign_ineq_2}, since $\Re F(iy) + \Re F(2\sigma - 1 + iy) \ge 0$ by \eqref{f_one_sign_ineq_1}.
\end{proof}

Lemma \ref{f_one_sign_lem} is used to isolate the contribution of a pair of zeroes at height $t$, while discarding all other zeroes (incurring a small loss).

\begin{lemma}\label{F_sum_lem}
Suppose $\zeta(\beta_0 + it) = 0$ with $\beta_0 > 1/2$ and $t \ge H$. Let $A \in (0, A_0]$, $\eta \in (0, A_0/\log t]$ and $F$ be as defined in Definition \ref{wfF_defn}. Furthermore suppose that $\zeta$ has no zeroes $\beta + i\gamma$ in the region $\beta \ge 1 - A/\log \gamma$ and $\gamma \ge H$. Let $s_k = \sigma + ikt$ where $0 \le k \le K$ and $\sigma = 1 - A/\log (K t + T_0)$. Then 
\[
\Re \sum_{\rho}F(s_1 - \rho) \ge F(\sigma - \beta_0) + F(\sigma - 1 + \beta_0) - 10^{-8}.
\]
Furthermore, for any $k \ne 1$ one has $
\Re \sum_{\rho}F(s_k - \rho) \ge -10^{-8}.$
\end{lemma}
\begin{proof}
Throughout let us write
$T = Kt + T_0$.
We divide the zeroes $\rho = \beta + i\gamma$ in the critical strip into those with $|\gamma| \le T$ and those with $|\gamma| > T$. 
Suppose first that $|\gamma| \le T$. Then 
\[
\beta \le 1 - A/\log |\gamma| \le 1 - A / \log T = \sigma
\]
so that $\Re (s_k - \rho) \ge 0$ for all such zeroes. A similar argument shows that $\Re (s_k - \rho) \le 2\sigma - 1$. Also, if $\rho$ is a complex zero off the critical line, then $1 - \overline{\rho}$ is also a zero. On the other hand, if $\rho$ lies on the critical line then $s_k - \rho = s_k - (1 - \overline{\rho})$. Thus 
%
\begin{equation*}
\Re \sum_{\substack{\rho = \beta + i\gamma\\|\gamma|\le T}} F(s_k - \rho) = \frac{1}{2}\Re\sum_{\substack{\rho = \frac{1}{2} + i\gamma\\|\gamma|\le T}}F(s_k - \rho) + F(s_k - (1 - \overline{\rho}))  + \Re \sum_{\substack{\rho = \beta + i\gamma\\ \beta > \frac{1}{2} \\ |\gamma| \le T}}F(s_k - \rho) + F(s_k - (1 - \overline{\rho})).
\end{equation*}
First, suppose $k \ne 1$. Since $0 \le \Re (s_k - \rho) \le 2\sigma - 1$ we may apply \eqref{f_one_sign_ineq_1} termwise with $z = s - \rho$, so that both sums on the right side are non-negative. Therefore
\begin{equation}\label{sum_F_final_est_1}
\Re \sum_{\substack{\rho = \beta + i\gamma\\|\gamma|\le T}} F(s_k - \rho) \ge 0\qquad (k \ne 1).
\end{equation}
On the other hand if $k = 1$, then we isolate the term corresponding to $\rho = \rho_0$ in the second sum, and apply \eqref{f_one_sign_ineq_1} to all other terms (recall that $\rho_0 = \beta_0 + it$ is our hypothetical zero off the critical line). We obtain 
\begin{equation}\label{sum_F_final_est_2}
\Re \sum_{\substack{\rho = \beta + i\gamma\\|\gamma|\le T}} F(s_1 - \rho) \ge F(\sigma - \beta_0) + F(2\sigma - 1 + \beta_0).
\end{equation}
It remains to estimate the contribution of zeroes with $|\gamma| > T$. These are far from $s$ so the corresponding terms are small, and we can afford to estimate them crudely. Applying \eqref{f_one_sign_ineq_2} termwise, one obtains 
\begin{equation*}
\Re\sum_{\substack{\rho = \beta + i\gamma \\ |\gamma| > T}}F(s_k - \rho) = \Re \sum_{\substack{\rho = \beta + i\gamma\\\gamma > T}}F(s_k - \rho) + F(s_k - \overline{\rho}) \ge -44\eta \sum_{\substack{\rho = \beta + i\gamma \\ \gamma > T}}\left(\frac{1}{(\gamma - kt)^2} + \frac{1}{(\gamma + kt)^2}\right).
\end{equation*}
Here we have made use of the fact that $|\Im (s_k - \rho)|, |\Im (s_k - \overline{\rho})| > T_0$ for each zero in the sum. If $k = 0$ then we use \cite[Lemma 2]{lehman_difference_1966} (see also \cite{BPT,Bellotti}) with $n = 2$ to get
\[
\sum_{\gamma > T_0}\frac{1}{\gamma^2} < \frac{\log T_0}{T_0} 
\]
so that, since $\eta \le A_0/\log t\le \eta_0$ by assumption,
\begin{equation}\label{sum_F_final_est_3}
\Re \sum_{\substack{\rho = \beta + i\gamma \\ |\gamma| > T}}F(\sigma - \rho) > -88\eta_0 \sum_{\gamma > T_0}\frac{1}{\gamma^2} > -10^{-8}. 
\end{equation}
If $k \ne 0$, then we use the same treatment as \cite{HoffTrudgian}. Applying \cite[Lemma 1]{lehman_difference_1966} (see also \cite{BPT}) with $\phi(x) = (x-kt)^{-2} + (x+kt)^{-2}$, we get
\begin{align*}
\sum_{\gamma > T}\phi(\gamma) &= \frac{1}{2\pi}\int_{T}^{\infty}\phi(x) \log \frac{x}{2\pi}\text{d}x + O^*\left(4\phi(T)\log T + 2\int_{T}^\infty\frac{\phi(x)}{x}\text{d}x\right),
\end{align*}
where, here and henceforth, $O^*(A)$ means a complex number whose modulus does not exceed $A$.
Recall that $T = Kt + T_0$ and $k \le K$, so that 
\begin{align*}
&\int_{T}^{\infty}\phi(x) \log \frac{x}{2\pi}\text{d}x \le \int_{T_0}^\infty\left(\frac{1}{x^2} + \frac{1}{(x + 2kt)^2}\right)\log\frac{kt + x}{2\pi}\text{d}x \\
&\qquad\qquad = \frac{1}{T_0}\log\frac{kt + T_0}{2\pi} + \frac{1}{2kt + T_0}\log\frac{kt + T_0}{2\pi}  +\frac{\log (2kt + T_0)}{kt} - \frac{\log T_0}{kt}.
\end{align*}
We drop the last term, and bound the second and third terms using $kt \ge H$, since they are both decreasing in $kt$. Finally, the first term may be bounded using $\log (1 + x) \le x$, so that
\[
\log\frac{kt + T_0}{2\pi} \le \log \frac{Kt}{2\pi} \left(1 + \frac{T_0}{K t}\right) \le \log t + \frac{T_0}{KH} + \log \frac{K}{2\pi}.
\]
Computing the constants explicitly, one finds
\begin{equation}\label{sum_F_est_1}
\int_{T_0}^{\infty}\phi(x) \log \frac{x}{2\pi}\text{d}x < 10^{-10}(\log t + 1).
\end{equation}
Next, since $kt \ge H$, we have
\begin{equation}\label{sum_F_est_2}
\phi(T_0) \log T_0 \le \left(\frac{1}{T_0^2} + \frac{1}{(2H + T_0)^2}\right)\log(kt + T_0) < 10^{-10}(\log t + 1),
\end{equation}
\begin{equation}\label{sum_F_est_3}
\int_{T_0}^{\infty}\frac{\phi(x)}{x}\text{d}x \le \int_{T_0}^{\infty}\left(\frac{1}{x^2} + \frac{1}{(2H + x)^2}\right)\frac{\text{d}x}{x} < 10^{-10},
\end{equation}
where in deriving \eqref{sum_F_est_2} we used $\log (kt + T_0) \le \log t + T_0/(KH) + \log K$. 
Combining \eqref{sum_F_est_1}, \eqref{sum_F_est_2} and \eqref{sum_F_est_3}, and since $\eta \le A_0/\log t \le \eta_0$ by assumption
\begin{align*}
\Re\sum_{\substack{\rho = \beta + i\gamma \\ |\gamma| > T}}F(s_k - \rho) &> -44\eta\left(\frac{10^{-10}}{2\pi} + 6\cdot 10^{-11}\right)(\log t + 1) > -10^{-8}. 
\end{align*}
The result follows from combining the above inequality with \eqref{sum_F_final_est_1}, \eqref{sum_F_final_est_2} and \eqref{sum_F_final_est_3}.
\end{proof}

One can use the assumed properties of $f$ to obtain an ``explicit formula". First, recall the following theorem from \cite{Kadiri}.

\begin{lemma}[Kadiri \cite{Kadiri} Proposition 2.1]\label{kadiri_prop_21}
Let $d > 0$ and suppose $g$ is a real-valued function supported on $[0, d)$ and twice continuously differentiable on $[0, d]$, satisfying $g(d) = g'(0) = g'(d) = g''(d) = 0$. Let $G(z) := \int_0^{\infty}e^{-zu}g(u)\text{d}u$ and $G_2(z) := \int_0^\infty e^{-zu}g''(u)\text{d}u$ denote respectively the Laplace transforms of $g$ and $g''$. Then, for all complex numbers $s$,
\begin{align*}
&\Re \sum_{n \ge 1}\frac{\Lambda(n)}{n^s}g(\log n) = g(0)\left(-\frac{1}{2}\log \pi + \frac{1}{2}\Re\frac{\Gamma'}{\Gamma}\left(\frac{s}{2} + 1\right)\right) + \Re \,G(s - 1) \\
&\qquad - \Re \sum_{\rho}G(s - \rho) + \Re\bigg(\frac{1}{2\pi i}\int_{1/2-i\infty}^{1/2 + i\infty}\frac{G_2(s - z)}{(s-z)^2}\Re\frac{\Gamma'}{\Gamma}\left(\frac{z}{2}\right)\text{d}z  + \frac{G_2(s)}{s^2}\bigg).
\end{align*}
\end{lemma}

The next lemma is used to control the error term in the case relevant to our application.
\begin{lemma}[Bounding the error term]\label{classical_explicit_formula_error_bound}
Let $s = \sigma + it$, $\sigma \ge \sigma_0$ and $\eta \in (0, \eta_0]$. Assume either $t = 0$ or $t \ge H$ and $\eta \le A_0/\log t$. If $W_0(z) = W(z) - w(0)/z$ then 
\[
\left|\frac{1}{2\pi i}\int_{1/2-i\infty}^{1/2 + i\infty}W_0\Big(\frac{s - z}{\eta}\Big)\Re\frac{\Gamma'}{\Gamma}\left(\frac{z}{2}\right)\text{d}z + W_0\Big(\frac{s}{\eta}\Big)\right| \le \begin{cases}
623 \eta^3,& t = 0,\\
14\eta^2 + 424\eta^3,& t \ge H.
\end{cases}
\]
\end{lemma}
\begin{proof}
Let $E_1(s)$ denote the left side of the above inequality, so that
\begin{align*}
E_1(s) &\le \frac{1}{2\pi} \int_{-\infty}^{\infty}\left|W_0\left(\frac{s - (1/2 + iy)}{\eta}\right)\Re\frac{\Gamma'}{\Gamma}\left(\frac{1/2 + iy}{2}\right)\right|\text{d}y + |W_0(s /\eta)|\\
&= T_1 + T_2,
\end{align*}
say. For convenience we write $x_0 := \sigma_0 - 1/2$. By assumption, $\eta \le \eta_0$. If $\xi = (s - 1/2 - iy)/\eta$ then $|\xi| \ge \Re\,\xi \ge x_0/\eta_0$. Applying Lemma \ref{W0_bound_lem} (and noting that $C(x_0/\eta_0, x_0/\eta_0) < 52$), one has
\begin{equation}\label{classical_explicit_formula_e0}
\left|W_0\left(\frac{s - 1/2 - iy}{\eta}\right)\right| \le \frac{52\eta^3}{(x_0^2 + (t - y)^2)^{3/2}}.
\end{equation}
Meanwhile, as in \cite[Lemma 3.6]{Kadiri} we use 
\begin{equation}\label{est_digamma}
    \left|\Re\frac{\Gamma'}{\Gamma}\left(\frac{1/2 + iy}{2}\right)\right| \le U(y) := \begin{cases}
\displaystyle\frac{1}{2}\log\frac{16}{1 + 4y^2} + \frac{2}{1 + 4y^2} + 2,& |y| < 1/2,\\
\displaystyle\left|\log\frac{|y|}{2} - \frac{2}{1+4y^2}\right| + \frac{2}{3|y|} + \frac{1}{8y^2},& |y| \ge 1/2.
\end{cases}
\end{equation}
If $t = 0$ then via direct computation, 
\begin{equation}\label{classical_explicit_formula_e1}
\int_{-\infty}^{\infty}\frac{U(y)}{(x_0^2 + y^2)^{3/2}}\text{d}y < 24.
\end{equation}
On the other hand if $t \ge H$ then we divide the integral as
\[
\int_{-\infty}^{\infty}\frac{U(t + y)}{(x_0^2 + y^2)^{3/2}}\text{d}y = \int_{-\infty}^{-(t - 2)} + \int_{-(t - 2)}^{t - 2} + \int_{t - 2}^{\infty} = I_1 + I_2 + I_3.
\]
One may verify that $U(x) < \log(100 - x)$ for $x \le 2$, so that 
\[
I_1 < \int_{H - 2}^{\infty}\frac{\log(y + 100)}{y^3}\text{d}y < 10^{-10}.
\]
For $x \ge 2$, one may check that $U(x) \le \log x$ so 
\begin{align*}
I_2 &\le \int_{-(t - 2)}^{t - 2}\frac{\log (t + y)}{(x_0^2 + y^2)^{3/2}}\text{d}y < \log t \int_{|y| \le t - 2}\frac{\text{d}y}{(x_0^2 + y^2)^{3/2}} < \frac{2}{x_0^2}\log t
\end{align*}
and also 
\[
I_3 \le \int_{H - 2}^{\infty}\frac{\log y}{y^3}\text{d}y < 10^{-10}.
\]
Combining these estimates with \eqref{classical_explicit_formula_e0} and \eqref{classical_explicit_formula_e1}, and using $\eta \le A_0/\log t \le \eta_0$ one has 
\begin{equation}\label{E_bound_term1}
T_1 \le \frac{52}{2\pi}\eta^3 \int_{-\infty}^{\infty}\frac{U(y)}{(x_0^2 + y^2)^{3/2}}\text{d}y < \begin{cases}
199\eta^3,& t = 0,\\
14\eta^2,& t \ge H.
\end{cases}
\end{equation}
Next if $\xi = s/\eta$ then $|\xi| \ge \Re\,\xi \ge \sigma_0/\eta_0$ so that by Lemma \ref{W0_bound_lem}, for all $k \ge 0$
\begin{equation}\label{E_bound_term2}
T_2 \le \frac{C(\sigma_0/\eta_0, \sigma_0/\eta_0)}{x_0^3}\eta^3 < 424\eta^3. 
\end{equation}
The desired result follows from combining \eqref{E_bound_term1} and \eqref{E_bound_term2}.
\end{proof}
\begin{remark}
The bounds for the digamma function in \eqref{est_digamma} could be sharpened slightly by estimating more carefully the terms appearing in the corresponding asymptotic expansions in each case. However, such refinements would yield only a negligible improvement in the final constant for the zero-free region. We therefore keep the present estimates for simplicity.
\end{remark}
\begin{lemma}[The explicit formula]\label{explicit_formula_lem}
Let $s = \sigma + it$ with $\sigma \in [\sigma_0, 1)$ and $t \ge H$, and suppose $\eta \in (0, A_0/\log t]$. If $f$ and $F$ are as in Definition \ref{wfF_defn}, then 
\begin{align*}
\Re \sum_{n \ge 1}\frac{\Lambda(n)}{n^s}f(\log n) &= f(0)\sum_{0 \le m \le M}\kappa_m\left(-\frac{1}{2}\log \pi + \frac{1}{2}\Re\frac{\Gamma'}{\Gamma}\left(\frac{s + \delta_m}{2} + 1\right)\right) \\
&\qquad + \Re F(s - 1) - \Re \sum_{\rho}F(s - \rho) + E,
\end{align*}
where 
\[
|E| < \begin{cases}
2353\eta^3,& t = 0,\\
53\eta^2 + 1601\eta^3,& t \ge H.
\end{cases}
\]
\end{lemma}
\begin{proof}
Note that if $g(u) = \eta w(\eta u)$ then $g(d) = g'(0) = g'(d) = g''(d)$ and
\[
\frac{G_2(z)}{z^2} = G(z) - \frac{g(0)}{z} = W_0(z/\eta).
\]
We apply Lemma \ref{explicit_formula_lem} with this choice of $g$ and with $s$ replaced by $s + \delta_m$ ($0 \le m \le M$):
\begin{align*}
\Re\sum_{n \ge 1}&\frac{\Lambda(n)}{n^s}f(\log n) = \Re \sum_{0 \le m \le M}\kappa_m\sum_{n \ge 1}\frac{\Lambda(n)}{n^{s + \delta_m}}\eta w(\eta \log n) \\
&=\Re \sum_{0 \le m \le M}\kappa_m\bigg(\eta w(0)\left(-\frac{1}{2}\log \pi + \frac{1}{2}\frac{\Gamma'}{\Gamma}\left(\frac{s + \delta_m}{2} + 1\right)\right) \\
&\qquad\qquad + W\left(\frac{s + \delta_m - 1}{\eta}\right) - \sum_{\rho}W\left(\frac{s + \delta_m - \rho}{\eta}\right) + E_1(s + \delta_m)\bigg)\\
&= f(0) \Re \sum_{0 \le m \le M}\kappa_m\left(-\frac{1}{2}\log \pi + \frac{1}{2}\frac{\Gamma'}{\Gamma}\left(\frac{s + \delta_m}{2} + 1\right)\right) \\
&\qquad\qquad + \Re F(s - 1) - \Re \sum_{\rho}F(s - \rho) + \sum_{0 \le m \le M}\kappa_m E_2(s + \delta_m),
\end{align*}
where 
\[
E_2(s) := \Re\frac{1}{2\pi i}\int_{1/2-i\infty}^{1/2 + i\infty}W_0\left(\frac{s - z}{\eta}\right)\Re\frac{\Gamma'}{\Gamma}\left(\frac{z}{2}\right)\text{d}z + \Re W_0(s/\eta).
\]
Applying Lemma \ref{classical_explicit_formula_error_bound} (with $s$ replaced by $s + \delta_m$ ($0 \le m \le M$) as required), and using the values of $\kappa_m$, one has
\[
\sum_{0 \le m \le M}|\kappa_m E_2(s + \delta_m)| < \begin{cases}
2353\eta^3,& t = 0,\\
53\eta^2 + 1601\eta^3,& t \ge H,
\end{cases}
\]
as required. 
\end{proof}

\section{The trigonometric polynomial}\label{sec:trig_poly}
As we briefly mentioned earlier, we require a non-negative trigonometric polynomial $
P(x) := \sum_{0 \le k \le K}a_k \cos kx \ge 0$
with coefficients satisfying $a_0 < a_1$ and $a_k \ge 0$ $(0 \le k \le K)$. 
We will use the following choice, which is a rational version of the polynomial in \cite[Table 5]{HoffTrudgian}
\[
P(x) := \frac{1}{14912370}\bigg|\sum_{0 \le k \le K}c_k e^{ik x}\bigg|^2 ,
\]
where we recall that $K = 16$, $c_0 = 4$ and
\begin{alignat*}{4}
&c_1 = -8,\qquad\qquad\quad
&&c_2 = 2,\qquad\qquad\quad
&&c_3 = 20,\qquad\qquad\quad
&&c_4 = -9,\\
&c_5 = -34,
&&c_6 = 27,
&&c_7 = 91,
&&c_8 = -27,\\
&c_{9} = -201,
&&c_{10} = 32,
&&c_{11} = 895,
&&c_{12} = 1949,\\
&c_{13} = 2389,
&&c_{14} = 1896,
&&c_{15} = 949,
&&c_{16} = 239.
\end{alignat*}
One may verify computationally that all conditions on $P(x)$ are satisfied. In particular 
\begin{equation}\label{classical_trig_poly_defn}
a_0 = 1,\qquad a_1 = \frac{865534}{497079},\qquad a := \sum_{1 \le k \le K}a_k = \frac{2919857}{828465}.
\end{equation}
In standard proofs of the zero-free region, the non-negativity of $P(x)$ is used to establish 
\[
\sum_{n \ge 1}\frac{\Lambda(n)}{n^\sigma}f(\log n)P(t\log n) \ge 0.
\]
In the remainder of this section, we show that it is possible to strengthen this inequality by replacing the right side with a positive quantity. This improvement will be used later to obtain a larger zero-free region.

\begin{lemma}\label{classical_main_trig_lower_bound_lem}
Let $f(u)$ be defined in Definition \ref{wfF_defn} and $P(x)$ be defined as above. Then 
\[
\sum_{n \ge 1}\frac{\Lambda(n)}{n^\sigma}f(\log n)P(t\log n) \ge 0.1186f(0).
\]
\end{lemma}

Our main idea is that instead of applying the estimate $P(t\log n) \ge 0$ pointwise, we first sum over terms  $n = p^m$ ($m \ge 1$) where $p$ is a fixed prime. In particular, we only consider  primes less than $100$ since the contributions from the other primes are negligible. 
\begin{lemma}\label{classical_f_lower_est}
If $f$ is as defined in Definition \ref{wfF_defn}, then $f(u) \ge \kappa f(0)$ for $0 \le u \le 59$, where $\kappa = \sum_{0 \le m \le M}\kappa_m$.
\end{lemma}
\begin{proof}
Note that $w'(x) \le 0$ for $x \ge 0$, so $w(\eta u) \ge w(\eta_0 u)$. Meanwhile, $2\sigma_0 - 1 \le \delta_m \le m$ so
\[
\eta^{-1}f(u) = w(\eta u) \sum_{0 \le m \le M}\kappa_m e^{-\delta_m u} \ge w(\eta_0 u) \sum_{0 \le m \le M}\kappa_m e^{-d_m u},
\]
where $d_m = (2\sigma_0 - 1)m$ if $\kappa_m < 0$ and $d_m = m$ otherwise. The right side is a pure function of $u$, and a routine computer-assisted computation is used to verify that 
\[
- \kappa w(0) + w(\eta_0 u) \sum_{0 \le m \le M}\kappa_m e^{-d_m u} \ge 0
\]
for $0 \le u \le 59$. The result follows from multiplying by $\eta$. 
\end{proof}
Using the non-negativity of $f(u)$ and $P(x)$, we may discard all terms with $n>N:=e^{59}$. Together with Lemma \ref{classical_f_lower_est}, this gives
$$
\sum_{n \geq 1} \frac{\Lambda(n)}{n^\sigma} f(\log n) P(t \log n) \geq \kappa f(0) \sum_{1 \leq n \leq N} \frac{\Lambda(n)}{n^\sigma} P(t \log n) .
$$
Expanding over prime powers, we obtain
$$
\sum_{n \geq 1} \frac{\Lambda(n)}{n^\sigma} f(\log n) P(t \log n) \geq \kappa f(0) \sum_p \log p \sum_{1 \leq m \leq \log N / \log p} p^{-m \sigma} \sum_{0 \leq k \leq K} a_k \cos (m k t \log p) .
$$
Since $P(x) \geq 0$, we may further discard all primes $p>100$ and all terms with $m>15$, since a numerical verification shows that their contribution is negligible. Hence
$$
\sum_{n \geq 1} \frac{\Lambda(n)}{n^\sigma} f(\log n) P(t \log n)>\kappa f(0) \sum_{p < 100} \log p \sum_{0 \leq k \leq K} a_k \sum_{1 \leq m \leq 15} p^{-m \sigma} \cos (m k t \log p) .
$$
Fix a prime $p$, and write $x=t \log p$. Denoting $\ell=m k$, we can rewrite the double sum 
\[
G_p(\sigma, x)=\sum_{0 \leq k \leq K} a_k \sum_{1 \leq m \leq 15} p^{-m \sigma} \cos (m k x)
\]
as
$$
G_p(\sigma, x)=\sum_{\ell \leq 15 K} b_{\ell}(\sigma) \cos (\ell x),
$$
where
$$
b_0(\sigma)=a_0 \sum_{m=1}^{15} p^{-m \sigma}=a_0 \frac{p^{-\sigma}\left(1-p^{-15 \sigma}\right)}{1-p^{-\sigma}},
$$
and, for $\ell \geq 1$,
$$
b_{\ell}(\sigma)=\sum_{\substack{1 \leq k \leq K \\ k \mid \ell \\ \ell / k \leq 15}} a_k p^{-\sigma \ell / k}.
$$
Equivalently,
$$
G_p(\sigma, x)=a_0 \sum_{m=1}^{15} p^{-m \sigma}+\sum_{\ell=1}^{15 K}\left(\sum_{\substack{1 \leq k \leq K \\ k \mid \ell \\ \ell / k \leq 15}} a_k p^{-\sigma \ell / k}\right) \cos (\ell x) .
$$
Differentiating with respect to $\sigma$, we obtain
$$
\frac{\partial G_p}{\partial \sigma}(\sigma, x)=-(\log p) a_0 \sum_{m=1}^{15} m p^{-m \sigma}-(\log p) \sum_{\ell=1}^{15 K}\left(\sum_{\substack{1 \leq k \leq K \\ k \mid \ell \\ \ell / k \leq 15}} a_k \frac{\ell}{k} p^{-\sigma \ell / k}\right) \cos (\ell x) .
$$
For each prime $p <100$, we verified numerically, using Mathematica, that
$$
\frac{\partial G_p}{\partial \sigma}(\sigma, x)<0 \quad\left(\sigma \in\left[\sigma_0, 1\right], x \in[0,2 \pi]\right) .
$$
Hence, for every fixed $x$, the function $\sigma \mapsto G_p(\sigma, x)$ is decreasing on $\left[\sigma_0, 1\right]$, and therefore
$
G_p(\sigma, x) \geq G_p(1, x).
$
It remains to find a lower bound for $G_p(1, x)$. Writing $y=\cos x$, we use the identity
$
\cos (\ell x)=T_{\ell}(y),
$
where $T_{\ell}$ is the $\ell$-th Chebyshev polynomial. Thus $G_p(1, x)$ becomes a polynomial $P_p(y)$ on $[-1,1]$. For each prime $p < 100$, we computed numerically the minimum of $P_p(y)$ on $[-1,1]$, truncated this minimum downward to the eighth decimal place, and then verified rigorously that
$$
P_p(y)-m_p>0, \quad(y \in[-1,1]),
$$
by exact real-algebraic verification in Mathematica. This yields a certified lower bound
$
G_p(1, x) > m_p.
$
For all primes $p<100$, the resulting certified lower bounds are displayed in Table \ref{tab:prime-lower-bounds}.
\begin{table}[ht]
\centering
\begin{tabular}{|c|c|c|c|}
\hline
$p$ & Certified lower bound $m_p$ & $p$ & Certified lower bound $m_p$ \\
\hline
2  & 0.23416332          & 43 & $4.91\cdot 10^{-6}$ \\
3  & 0.05006283          & 47 & $3.74\cdot 10^{-6}$ \\
5  & 0.00735709          & 53 & $2.59\cdot 10^{-6}$ \\
7  & 0.00205251          & 59 & $1.87\cdot 10^{-6}$ \\
11 & 0.00039742          & 61 & $1.69\cdot 10^{-6}$ \\
13 & 0.00022273          & 67 & $1.27\cdot 10^{-6}$ \\
17 & 0.00009071          & 71 & $1.06\cdot 10^{-6}$ \\
19 & 0.00006314          & 73 & $9.7\cdot 10^{-7}$ \\
23 & 0.00003427          & 79 & $7.7\cdot 10^{-7}$ \\
29 & 0.00001656          & 83 & $6.6\cdot 10^{-7}$ \\
31 & 0.00001346          & 89 & $5.3\cdot 10^{-7}$ \\
37 & $7.8\cdot 10^{-6}$  & 97 & $4.1\cdot 10^{-7}$ \\
41 & $5.69\cdot 10^{-6}$ &    &                      \\
\hline
\end{tabular}
\caption{Certified truncated lower bounds $m_p$ for $\min_x G_p(1,x)$, for primes $p<100$.}
\label{tab:prime-lower-bounds}
\end{table}
Summing these contributions, we obtain
$$
\sum_{p <100}(\log p) m_p>0.23545 .
$$
Therefore,
$$
\sum_{n \geq 1} \frac{\Lambda(n)}{n^\sigma} f(\log n) P(t \log n) \geq \kappa f(0) \sum_{p <100}(\log p) m_p>0.23545 \kappa f(0) .
$$
Using the chosen value of $\kappa$, this gives
$$
\sum_{n \geq 1} \frac{\Lambda(n)}{n^\sigma} f(\log n) P(t \log n)>0.1186 f(0),
$$
which proves Lemma \ref{classical_main_trig_lower_bound_lem}.

\section[Proof of Lemma 2.1]{Proof of Lemma \ref{classical_main_lem}}\label{castle}
To complete the argument we require a few more results to estimate certain terms appearing on the right side of Lemma \ref{explicit_formula_lem}; these are recorded in the next few lemmas. 

\begin{lemma}\label{digamma_lem}
Let $s_k = \sigma + ikt$ with $k \ge 0$ an integer, $\sigma_0 \le \sigma < 1$ and $t \ge H$. If $\kappa_m$ $(0 \le m \le M)$ and $a_k$ ($0 \le k \le K$) are as defined in \eqref{kappa_delta_defn} and Section \ref{sec:trig_poly} respectively, then 
\[
\sum_{0 \le k \le K}a_k\sum_{0 \le m \le M}\kappa_m\left(-\frac{1}{2}\log \pi + \frac{1}{2}\Re\frac{\Gamma'}{\Gamma}\left(\frac{s_k + \delta_m}{2} + 1\right)\right) \le \frac{a\kappa}{2}\log t - 1.568
\]
where $a := \sum_{1 \le k \le K}a_k = 2919857/828465$, $\kappa := \sum_{0 \le m \le M}\kappa_m = 433/859$.
\end{lemma}
\begin{proof}
For $k \ge 1$, we use the following estimate, due to \cite[(24)]{Kadiri}: if $z = x + iy$ with $0 < x < |y|$, then 
\begin{equation}\label{apples}
\Re\frac{\Gamma'}{\Gamma}(z) = \log |y| - \frac{1}{2}\frac{x}{|z|^2} + O^*\left(\frac{1}{12x|y|} + \frac{x^2}{2y^2}\right). 
\end{equation}
Applying (\ref{apples}) with $z = (s_k + \delta_m)/2 + 1$, all error terms for $k \ge 1$ are $O(t^{-1})$ so we can afford to estimate them roughly. In particular, for $1 \le x \le (M + 3)/2$ and $y \ge H/2$ one has 
\[
\left|- \frac{1}{2}\frac{x}{|z|^2} + \frac{1}{12x|y|} + \frac{x^2}{2y^2}\right| < 10^{-10}. 
\]
Therefore
\begin{equation}
\label{ch1:s0:e1}
\begin{split}
&\sum_{1 \le k \le K}a_k\sum_{0 \le m \le M}\kappa_m\left(-\frac{1}{2}\log \pi + \frac{1}{2}\Re\frac{\Gamma'}{\Gamma}\left(\frac{s_k + \delta_m}{2} + 1\right)\right) \\
&\qquad\qquad\le \frac{\kappa}{2}\sum_{1 \le k \le K}a_k\left(-\log \pi + \log \frac{kt}{2} + 10^{-10}\right) \\
&\qquad\qquad = \frac{a\kappa}{2}\log t + \frac{a\kappa}{2}(-\log 2\pi + 10^{-10}) + \frac{\kappa}{2}\sum_{1 \le k \le K}a_k \log k.
\end{split}
\end{equation}
For the term corresponding to $k = 0$ we use the fact that $\Gamma'/\Gamma(x)$ is increasing for $x > 0$, so
\[
\frac{\Gamma'}{\Gamma}\left(\frac{\sigma_0 + (2\sigma_0 - 1)m}{2} + 1\right) \le \frac{\Gamma'}{\Gamma}\left(\frac{\sigma + \delta_m}{2} + 1\right) \le \frac{\Gamma'}{\Gamma}\left(\frac{m + 3}{2}\right).
\]
Therefore, by choosing the appropriate bound depending on the sign of $\kappa_m$, we find 
\begin{equation}\label{ch1:s0:e2}
\frac{1}{2}\sum_{0 \le m \le M}\kappa_m\frac{\Gamma'}{\Gamma}\left(\frac{\sigma + \delta_m}{2} + 1\right) \le -0.041.
\end{equation}
The result follows from combining \eqref{ch1:s0:e1} and \eqref{ch1:s0:e2}.
\end{proof}

\begin{lemma}\label{classical_zeta_pole_lem}
Suppose $s_k = \sigma + ikt$ $(0 \le k \le K)$ with $\sigma_0 \le \sigma < 1$ and $t \ge H$. Let $\eta \in (0, \eta_0]$ and $F$ be as defined in Definition \ref{wfF_defn}. Then 
\[
\Re\sum_{0 \le k \le K}a_k F(s_k - 1) \le a_0 F(\sigma - 1) + 10^{-10}\eta.
\]
\end{lemma}
\begin{proof}
Suppose that $k \ge 1$. By Lemma \ref{W0_bound_lem}, for $\Re\,z \ge \nu, |z| \ge r$, one has
\[
\Re W(z) = \frac{w(0)\Re\,z}{|z|^2} + O^*(C(\nu, r)|z|^{-3}).
\]
Taking $z = (s_k + \delta_m - 1)/\eta$, so that $|z| \ge t/\eta \ge H/\eta_0 > 10^{10}$ and $\Re\,z \ge (\sigma - 1)/\eta > -1$. A computer verification gives $C(-1, 10^{10}) < 51$ which, together with $\Re\,z / |z|^2 < M\eta /H^2$, gives
\begin{align*}
\left|\Re W\left(\frac{s_k + \delta_m - 1}{\eta}\right)\right| < \frac{w(0) M}{H^2}\eta + \frac{51}{H^3}\eta^3 < 10^{-12}\eta.
\end{align*}
Therefore, for all $k \ge 1$,
\begin{align*}
\Re F(s_k - 1) &\le \left|\sum_{0 \le m \le M}\kappa_m W\left(\frac{s_k + \delta_m - 1}{\eta}\right)\right| < 10^{-11}\eta.
\end{align*}
The result follows from summing over $k$. 
\end{proof}

\begin{lemma}\label{classical_second_last_lem}
Suppose $\sigma \in [\sigma_0, 1)$, $\eta \in (0, \eta_0]$ and
\[
\frac{1 - \sigma_0}{\eta_0} - 10^{-10} \le \frac{1 - \sigma}{\eta} \le 1.
\]
Let $F$ be as defined in Definition \ref{wfF_defn} and $a_0$, $a_1$ be as defined in \eqref{classical_trig_poly_defn}. Then 
\begin{align*}
a_1 F(\sigma - 1 + \eta) + a_1 F(\sigma - \eta) &- a_0 F(\sigma - 1) \\
&\ge C_1((1 - \sigma)/\eta) + 3.909\eta + 26\eta^2 - 3897\eta^3,
\end{align*}
where $C_1(x) := 0.87637 + 0.12002 x + 0.01017 x^2 - 0.00073 x^3$.
\end{lemma}
\begin{proof}
For convenience, denote
\begin{equation}\label{classical_mu0_defn}
\mu := \frac{1 - \sigma}{\eta},\qquad \mu_0 := \frac{1 - \sigma_0}{\eta_0} - 10^{-10} = 0.91198\ldots.
\end{equation}
First we bound the expression $a_1 F(\sigma - 1 + \eta) - a_0 F(\sigma - 1)$, which may be written as
\[
\sum_{0 \le m \le M}\kappa_m(a_1 W(\delta_m/\eta - \mu + 1) - a_0 W(\delta_m/\eta - \mu)).
\]
For any $a, \varepsilon$, with $a \ne 0$ and $a + \varepsilon \ne 0$ one has
\begin{equation}\label{classical_temp_ineq}
\frac{1}{a + \varepsilon} = \frac{1}{a} - \frac{\varepsilon}{a^2} + \frac{\varepsilon^2}{a^2(a + \varepsilon)}
\end{equation}
so that
\begin{align*}
W(\delta_m/\eta + x) &= \frac{\eta w(0)}{\delta_m + x\eta} + W_0\left(\frac{\delta_m}{\eta} + x\right) = \frac{w(0)}{\delta_m}\eta - \frac{x w(0)}{\delta_m^2}\eta^2 + E(\eta)
\end{align*}
where
\[
E(\eta) = \frac{x^2 w(0)}{\delta_m^2(\delta_m + x\eta)}\eta^3 + W_0\left(\frac{\delta_m}{\eta} + x\right).
\]
If $|x| \le x_0$ and $\delta_m + x\eta \ge x_1 > 0$, then by Lemma \ref{W0_bound_lem},
\[
|E(\eta)| \le \left(\frac{x_0^2 w(0)}{(2\sigma_0 - 1)^2m^2 x_1} + \frac{C(x_1, x_1)}{x_1^3}\right)\eta^3.
\]
In particular, if $m \ge 1$ and $x = -\mu$ then we may take $x_0 = \mu_0$ and $x_1 = (2\sigma_0 - 1)m - \mu_0\eta_0$ so that, writing $S_k := \sum_{1 \le m \le M}\kappa_m/m^k$, we have
\[
\sum_{1 \le m \le M}\kappa_m W(\delta_m /\eta - \mu) \le \frac{w(0)\eta}{2\sigma - 1}S_1 - \frac{\mu w(0)\eta^2}{(2\sigma - 1)^2}S_2 + 1126 \eta^3.
\]
Similarly, if $x = -\mu + 1$ we take $x_0 = 1 - \mu_0$ and $x_1 = (2\sigma_0 - 1)m - (1 - \mu_0)\eta_0$ to get
\[
\sum_{1 \le m \le M}\kappa_m W(\delta_m /\eta - \mu + 1) \ge \frac{w(0)\eta}{2\sigma - 1}S_1 - \frac{(\mu + 1)w(0)\eta^2}{(2\sigma - 1)^2}S_2 - 1188 \eta^3.
\]
We explicitly evaluate $S_1 = -0.689736127\ldots$ and $S_2 = -0.818779265\ldots$. Furthermore, substituting the values of $a_0$, $a_1$ and $w(0)$, and applying $\mu \ge \mu_0$ and $\sigma \ge \sigma_0$, one finds that
\begin{equation}\label{classical_main_term_bound1}
\begin{split}
\sum_{1 \le m \le M}\kappa_m &(a_1 W(\delta_m/\eta - \mu + 1) - a_0 W(\delta_m/\eta - \mu))\\
&> \eta\frac{w(0)(a_1 - a_0)}{2\sigma - 1}S_1 - \eta^2\frac{w(0)((a_1 - a_0)\mu + a_1)}{(2\sigma - 1)^2}S_2 - 3188\eta^3\\
&> -2.939\eta + 11\eta^2 - 3193\eta^3.
\end{split}
\end{equation}
Next, the term corresponding to $m = 0$ is
\[
a_1 W(-\mu + 1) - a_0 W(-\mu) = \int_0^{2\theta\cot\theta}e^{\mu u}(a_1 e^{-u} - a_0)w(u)\text{d}u.
\]
For all $0 \le x \le 2\theta\cot\theta = 1.05923293\ldots$ one has 
\[
\bigg|1 + x + \frac{x^2}{2} + \frac{x^3}{6} - e^x\bigg| < \frac{1}{18}x^4
\]
so that
\[
a_1 W(-\mu + 1) - a_0 W(-\mu) > c_0 + c_1\mu + \frac{c_2}{2} \mu^2 + \left(\frac{c_3}{6} - \frac{c^*}{18}\right)\mu^3
\]
where
\[
c_n := \int_0^{2\theta\cot\theta} u^n(a_1e^{-u} - a_0)w(u)\text{d}u, \quad c^* := \int_0^{2\theta\cot\theta} u^3|a_1e^{-u} - a_0|w(u)\text{d}u
\]
are computable constants. In particular, we find, using computer assistance,
\begin{align*}
c_0 &= 0.8763706262\ldots,\quad c_1 = 0.1200272738\ldots, \quad c_2 = 0.0203537951\ldots,\\
c_3 &= 0.0004382722\ldots,\quad c^* = 0.0190417514\ldots,
\end{align*}
so that
\begin{equation}\label{classical_main_term_bound2}
a_1 W(-\mu + 1) - a_0 W(-\mu) > C_1(\mu).
\end{equation}
It remains to bound $a_1 F(\sigma - \eta)$. One has 
\[
F(\sigma - \eta) = \sum_{0 \le m \le M}\kappa_m W((\sigma - \eta + \delta_m)/\eta).
\]
Applying \eqref{classical_temp_ineq} with $a = 1 + \delta_m$ and $\varepsilon = \sigma - 1 - \eta = -(\mu + 1)\eta$, we have
\begin{align*}
W\left(\frac{\sigma + \delta_m - \eta}{\eta}\right) = \eta w(0)\left(\frac{1}{1 + \delta_m} + \frac{1 + \mu}{(1 + \delta_m)^2}\eta\right) + E_1(\eta),
\end{align*}
where
\[
E_1(\eta) = \frac{w(0)(1 + \mu)^2}{(1 + \delta_m)(1 + \delta_m - (1 + \mu)\eta)}\eta^3 + W_0\left(\frac{\sigma + \delta_m - \eta}{\eta}\right).
\]
Recall that $\mu_0 \le \mu \le 1$, $\eta \le \eta_0$ and $\delta_m \ge d_m := (2\sigma_0 - 1)m$. Furthermore, for all $m \ge 0$, one has $(\sigma + \delta_m - \eta)/\eta \ge \sigma_0/\eta_0 - 1$ so that by Lemma \ref{W0_bound_lem}
\begin{align*}
|E_1(\eta)| &\le \bigg(\frac{4w(0)}{(1 + d_m)(1 + d_m - (1 + \mu_0)\eta_0)m} + \frac{C(\sigma_0/\eta_0 - 1, \sigma_0/\eta_0 - 1)}{(\sigma_0 - \eta_0)^3}\bigg)\eta^3.
\end{align*}
Next we write $T_k(\sigma) := \sum_{0 \le m \le M}\kappa_m/(1 + \delta_m)^k$, so that
\[
T'_k(\sigma) = -2k\sum_{1 \le m \le M}\frac{m\kappa_m}{(1 + \delta_m)^{k + 1}}
\]
is positive for $k = 1, 2$ and all $\sigma_0 \le \sigma \le 1$. Therefore, 
\begin{align*}
a_1 F(\sigma - \eta) &= a_1 \eta w(0)\left(T_1(\sigma) + \eta (1 + \mu) T_2(\sigma)\right) + a_1 E_1(\eta)\\
&\ge a_1 \eta w(0)T_1(\sigma_0) + a_1 \eta^2 w(0)(1 + \mu_0) T_2(\sigma_0) + a_1 E_1(\eta)\\
&> 6.848\eta + 15\eta^2 - 708\eta^3.
\end{align*}
The result follows from combining this bound with \eqref{classical_main_term_bound1} and \eqref{classical_main_term_bound2}. 
\end{proof}
We now proceed to the main inductive argument. 

\begin{lemma}[Iteration lemma]\label{classical_iteration_lem}
Suppose that $\zeta(\sigma + it) \ne 0$ in the region 
\[
\sigma > 1 - \frac{A}{\log t},\qquad t \ge H,
\]
for some $1/6 < A \le A_0$. Then $\zeta(\sigma + it) \ne 0$ in the region
\[
\sigma > 1 - \frac{A + \varepsilon}{\log t},\qquad t \ge H,
\]
with $\varepsilon = 10^{-100}$.
\end{lemma}
\begin{proof}
Suppose for a contradiction that $\rho_0 = \beta_0 + it$ is a zero of $\zeta(s)$ with
\begin{equation}\label{classical_main_zero_assumption}
A \le (1 - \beta_0)\log t < A + \varepsilon\quad \text{and}\quad t \ge H. 
\end{equation}
Fix $s_k = \sigma + ikt$ with $0 \le k \le K$ and $\sigma = 1 - A/\log (Kt + T_0)$, so that $\sigma \in [\sigma_0, 1)$. Let
$\eta := 1 - \beta_0$
so that $\eta \in (0, \eta_0]$, and choose
\[
f(u) = \eta w(\eta u)\sum_{0 \le m \le M}\kappa_m e^{-\delta_m u}
\]
as in Definition \ref{wfF_defn}. By Lemma \ref{explicit_formula_lem} with $s = s_k$,
\begin{align*}
\Re\sum_{n \ge 1}\frac{\Lambda(n)}{n^{s_k}}f(\log n) = f(0)\sum_{0 \le m \le M}&\kappa_m\left(-\frac{1}{2}\log \pi + \frac{1}{2}\Re\frac{\Gamma'}{\Gamma}\left(\frac{s_k + \delta_m}{2} + 1\right)\right) \\
&+ \Re F(s_k - 1) - \Re \sum_{\rho}F(s_k - \rho) + E_k,
\end{align*}
where $|E_k| \le 2353\eta^3$ if $k = 0$ and $|E_k| \le 53\eta^2 + 1601\eta^3$ if $k \ge 1$. Next, let 
$P(x) = \sum_{0 \le k \le K}a_k \cos kx$
be the non-negative trigonometric polynomial constructed in Section \ref{sec:trig_poly}. By Lemma~\ref{classical_main_trig_lower_bound_lem} and Lemma \ref{explicit_formula_lem} one has 
\begin{equation}\label{classical_main_ineq}
\begin{split}
0.1186f(0)&\le \sum_{n \ge 1}\frac{\Lambda(n)}{n^\sigma}f(\log n)P(t\log n)\\
&=\Re \sum_{0 \le k \le K}a_k \sum_{n \ge 1}\frac{\Lambda(n)}{n^{s_k}}f(\log n) \\
&= \sum_{0 \le k \le K}a_k \bigg[f(0)\sum_{0 \le m \le M}\kappa_m\left(-\frac{1}{2}\log \pi + \frac{1}{2}\Re\frac{\Gamma'}{\Gamma}\left(\frac{s_k + \delta_m}{2} + 1\right)\right) \\
&\qquad\qquad\qquad + \Re F(s_k - 1) - \Re \sum_{\rho}F(s_k - \rho) + E_k\bigg].
\end{split}
\end{equation}
Now we invoke a series of estimates to bound the terms on the right side. First, Lemma~\ref{digamma_lem} is used to bound the main term of order $\log t$:
\begin{equation}\label{classical:b3}
\sum_{0 \le m \le M}\kappa_m\sum_{0 \le k \le K}a_k\left(-\frac{1}{2}\log \pi + \frac{1}{2}\Re\frac{\Gamma'}{\Gamma}\left(\frac{s_k}{2} + 1\right)\right) \le \frac{a\kappa}{2}\log t - 1.568,
\end{equation}
where $\kappa := \sum_{0 \le m \le M}\kappa_m = 433/859$. Next, by Lemma \ref{classical_zeta_pole_lem} one has
\begin{equation}\label{classical:b1}
\Re\sum_{0 \le k \le K}a_k F(s_k - 1) \le a_0 F(\sigma - 1) + 10^{-10}\eta
\end{equation}
and by Lemma \ref{F_sum_lem} one has
\begin{equation}\label{classical:b2}
\Re\sum_{0 \le k \le K}a_k \sum_{\rho}F(s_k - \rho) > F(\sigma - \beta_0) + F(\sigma - 1 + \beta_0) - 10^{-7}.
\end{equation}
Finally, Lemma \ref{explicit_formula_lem} gives a bound for the error term
\begin{equation}\label{classical:b4}
\bigg| \sum_{0 \le k \le K}a_k E_k \bigg| < 187\eta^2 + 8013\eta^3. 
\end{equation}
Substituting \eqref{classical:b1}, \eqref{classical:b2}, \eqref{classical:b3} and \eqref{classical:b4} into \eqref{classical_main_ineq}, we have 
\begin{align*}
0 &\le \frac{a\kappa}{2} f(0)\log t + a_0 F(\sigma - 1) - a_1 (F(\sigma - \beta_0) + F(\sigma - 1 + \beta_0))\\
&\qquad\qquad + 10^{-7} + (-0.1186w(0) - 1.568 w(0) + 10^{-10})\eta + 187\eta^2 + 8013\eta^3.
\end{align*}
At this point we note that by our assumption \eqref{classical_main_zero_assumption} (recalling that $\eta = 1 - \beta_0$, $A > 1/6$ and $\mu_0$ is as defined in \eqref{classical_mu0_defn}), we have
\begin{align*}
\mu := \frac{1 - \sigma}{\eta} > \frac{A}{A + \varepsilon}\frac{\log t}{\log(Kt + T_0)} > \mu_0
\end{align*}
and also $\mu < 1$, so that we may apply Lemma \ref{classical_second_last_lem} to obtain
\[
\frac{a\kappa}{2} f(0)\log t \ge C_1(\mu) + C_2(\eta) - 10^{-7}
\]
where $C_1$ is defined in Lemma \ref{classical_second_last_lem} and $C_2(\eta) := 13.47\eta - 161\eta^2 - 11896\eta^3$. Note that $C_1$ and $C_2$ are positive and increasing on $[\mu_0, 1]$ and $[0, \eta_0]$ respectively. Thus we may use\footnote{We note that some iteration here yields minor improvements.  Once we have a zero-free region constant in Theorem \ref{intro_new_theorem1} of $A_{0} = 1/4.9$, say, we can we apply this lemma with a marginally-improved lower bound for $\eta$ and hence with an improved $C_{2}$. We have not pursued this, since the improvements appear negligible.} (in light of $A > 1/6$)
\[
\mu > \frac{A}{A + \varepsilon}\frac{\log t}{\log (K + T_0/H) + \log t} > 1 - \frac{2.78}{\log t}\qquad\text{and}\qquad\eta > \frac{1}{6\log t}
\]
to obtain (with the aid of any symbolic algebra package)
\begin{align*}
\frac{a\kappa}{2} f(0)\log t &> C_1\left(1 - \frac{2.78}{\log t}\right) + C_2\left(\frac{1}{6\log t}\right) - 10^{-7}\\
&> 1.00582 + \frac{1.86088}{\log t} - \frac{4.4106}{(\log t)^2} - \frac{55.0584}{(\log{t})^{3}}.
\end{align*}
Since the right side is decreasing for $t \in [H, \exp(76.47)]$, we may lower-bound this expression by its value at $t = \exp(76.47)$, which is $1.02928\ldots$. Therefore
\[
\eta\log t > \frac{1.02928}{a \kappa w(0)/2} = 0.204248\ldots > A_0 + \varepsilon \ge A + \varepsilon,
\]
as required.
\end{proof}

To complete the proof of Lemma \ref{classical_main_lem} it remains to note that all zeroes $\beta + i\gamma$ with $3 \le \gamma < H$ are known to lie on the critical line \cite{platt_riemann_2021}, so Lemma \ref{classical_main_lem} follows in this region. For $\gamma \ge H$, we use Lemma \ref{classical_iteration_lem} combined with any existing zero-free region of the form $\sigma > 1 - c/\log t$ for some $c > 1/6$, e.g.\ those proved in \cite{Kadiri}, \cite{HoffTrudgian} or \cite{mossinghoff_explicit_2022}.


\end{document}